\documentclass[aap,MSNbibl,seceqn,citesort,dvips]{arximspdf}
\usepackage{mathbh}

\doi{10.1214/11-AAP815} 
\volume{22}
\issue{5}
\pubyear{2012}
\firstpage{1860}
\lastpage{1879}

\makeatletter
\newcommand{\eqref}[1]{(\ref{#1})}

\newcommand{\etas}{{\eta^*_{\ell}}}
\newcommand{\defeq}{\stackrel{\mathrm{def}}{=}}

\newtheorem{thmm}{Theorem}[section]

\newcommand{\pr}{{\mathbb{P}}}
\newcommand{\ex}{{\mathbb{E}}}
\newcommand{\openone}{\mathbh{1}}
\newcommand{\bls}{{\lfloor{\ell/s}\rfloor}}
\newcommand{\cls}{{\lceil{\ell/s}\rceil}}

\newtheorem{lem}{Lemma}[section]

\newtheorem{cor}{Corollary}[section]
\makeatother

\begin{document}
\begin{frontmatter}

\title{Tracking a random walk first-passage time through~noisy observations}
\runtitle{Tracking stopping times}

\begin{aug}
\author[A]{\fnms{Marat V.} \snm{Burnashev}\thanksref{t1}\ead[label=e1]{burn@iitp.ru}}
\and
\author[B]{\fnms{Aslan} \snm{Tchamkerten}\corref{}\thanksref{t2}\ead[label=e2]{aslan.tchamkerten@telecom-paristech.fr}}
\thankstext{t1}{Supported by the Russian Fund for Fundamental
Research (project number 09-01-00536).}
\thankstext{t2}{Supported in part by an Excellence Chair Grant from
the French National Research Agency (ACE project).}
\runauthor{M. V. Burnashev and A. Tchamkerten}
\affiliation{Russian Academy of Sciences and Telecom ParisTech}
\address[A]{Institute for Information\\
\quad Transmission Problems\\
Russian Academy of Sciences\\
Moscow\\
Russia\\
\printead{e1}} 
\address[B]{Communications and Electronics Department\\
Telecom ParisTech\\
75634 Paris Cedex 13\\
France\\
\printead{e2}}
\end{aug}

\received{\smonth{5} \syear{2010}}
\revised{\smonth{3} \syear{2011}}

%
\begin{abstract}
Given a Gaussian random walk (or a Wiener
process), possibly with drift, observed
through noise, we
consider the problem of estimating its
first-passage time $\tau_\ell$ of a given level $\ell$ with a stopping
time~$\eta$ defined over the noisy observation process.

Main results are upper and lower bounds on the
minimum mean absolute deviation $\inf_\eta
\ex|\eta-\tau_\ell|$ which
become tight as $\ell\to\infty$. Interestingly, in this regime
the estimation error does not get smaller if we allow $ \eta$ to be an
arbitrary
function of the entire observation process, not necessarily a stopping time.

In the particular case where there is no drift, we show that it is
impossible to
track $\tau_\ell$: $\inf_\eta
\ex|\eta-\tau_\ell|^p=\infty$ for any $\ell>0$ and $p\geq1/2$.
\end{abstract}

%

\end{frontmatter}

\section{Introduction} The tracking stopping time (TST) problem, recently
introduced in~\cite{NT}, is formulated as
follows. Let $X=\{X_t\}_{t\geq0}$ be a stochastic
process and let $\tau$ be a stopping time defined over $X$.
A statistician has access to $X$ only through correlated
observations $Y=\{Y_t\}_{t\geq0}$ and
wishes to find a~stopping $\eta$ that gets close
to $\tau$, for instance, so as to
minimize the average absolute deviation $\ex|\eta-\tau|$.
For specific
applications of the TST problem formulation related to monitoring,
forecasting and communication, we refer to~\cite{NT}.

In~\cite{NT}, an algorithmic solution is proposed for
discrete-time settings where the $(X_t,Y_t)$'s take on
values in a common finite alphabet (otherwise $X$ and~$Y$
are arbitrary processes) and where $\tau$ is bounded.
What motivated an algorithmic approach to this
problem is
that the TST problem generalizes the Bayesian change-point detection
problem, a
long-studied problem that dates back to the $1940$s, and for which
nonasymptotic solutions are known to be hard to obtain.

In the Bayesian change-point problem, there is a random variable
$\theta$, taking on values in the positive integers, and two
probability distributions $P_0$ and~$P_1$. Under $P_0$, the
conditional density function of $Z_t$ given
$Z_1,Z_2,\ldots,Z_{t-1}$ is $f_0(Z_t|Z_1,Z_2,\ldots,Z_{t-1})$, for
every $t\geq0$. Under $P_1$, the conditional density function of
$Z_t$ given $Z_1,Z_2,\ldots,Z_{t-1}$ is
$f_1(Z_t|Z_1,Z_2,\ldots,Z_{t-1})$, for every $t\geq0$. The
observed process $Y=\{Y_t\}_{t\geq0}$ is distributed according to
$P_0$ for all
$t<\theta$ and according to $P_1$
for all $t\geq\theta$. The problem typically consists in finding a
stopping time $\eta$, with respect to $\{Y_t\}$,
that is, close to $\tau$.

Nonasymptotic results for the Bayesian change-point problem
have been reported mostly for the i.i.d. case where, conditioned on the
change-point value, observations are independent with
common
distribution $P_0$
and $P_1$ before and after the change~\cite{Shi3,Shi}.\setcounter{footnote}{2}\footnote{An
exception is
\cite{Y} which considers Markov chains, but of finite state.}

The TST problem can be seen as a Bayesian change-point problem whose
change-point $\tau$ is a stopping time defined with respect to an unobserved
process~$X$ that depends on the observed
process $Y$. What specifically differentiates a
TST problem from a Bayesian change-point problem
is that for the latter we always have the
identity
\[
\pr(\theta=k|Y_0,Y_1,
\ldots,Y_n,k>n)=\pr(\theta=k|k>n),\qquad
k>n .
\]
In contrast, the above identity with $\theta=\tau$ need not hold for a
TST problem.
Because of this, past
observations are in general useful for estimating $\tau$. Furthermore, the
observed process $Y$ has usually memory once
conditioned on $\tau$.\footnote{Unless the TST problem
under consideration reduces to a
Bayesian change-point problem with independent
observations before and after the change.} This
is what makes the TST problem hard.

In this paper, we investigate
the natural setting case where $X$ is a Gaussian random
walk (or a Wiener process) possibly with drift, where $Y$ is a~noisy
version of
$X$, and where $\tau$ is the first time when $X$
reaches a given level~$\ell$. We establish a lower
bound on $\inf_\eta\ex|\eta-\tau|$, where the infimum is over all
stopping times with respect to $Y$, then exhibit a stopping rule
that achieves this bound as $\ell\to\infty$.
In the case where $X$ does not drift, we show that
$\ex|\eta-\tau|=\infty$ for any $\ell>0$ and any
estimator $\eta$, not necessarily a~stopping time.

Throughout the paper the following notational conventions are adopted.
We use~$\eta$ to denote a function of the
observation process $Y=Y_0^\infty$. When~$\eta$ has no argument,
we mean that $\eta$ is a stopping time with
respect to $Y$. Instead, if $\eta$ has an
argument, we mean that $\eta$ is a function of its
argument which need not be a stopping time with
respect to $Y$. For
example, $\eta(Y_t)$ refers
to a function of observation~$Y_{t}$.

Further, we frequently omit arguments of functions (or estimators)
that appear in
expressions to be optimized. For instance, instead of
\[
\inf_{\eta(Y_t)}\ex
|\eta(Y_t)-\tau_\ell|^p ,
\]
we simply write
\[
\inf_{\eta(Y_t)}\ex|\eta-\tau_\ell|^p
\]
to denote an optimization over estimators of
$\tau_\ell$ that depend only on observation~$Y_t$.

Section~\ref{sec:mainresult} contains the main results and
Section~\ref{sec:proofs} is devoted to the proofs.

\section{Main results}\label{sec:mainresult}
Consider the discrete-time processes
\begin{eqnarray*}
X\dvtx&&\quad  X_0=0,\qquad  X_t=\sum_{i=1}^t V_i +s t,\qquad t\geq1, \\
Y\dvtx&&\quad   Y_0=0,\qquad  Y_t=X_t+\varepsilon\sum_{i=1}^t W_i,\qquad
   t\geq1,
\end{eqnarray*}
where $V_1,V_2,\ldots$ and $W_1,W_2,\ldots$ are two independent
sequences of independent standard (i.e., zero-mean unit variance)
Gaussian random variables, and where \mbox{$s\geq0$} and $\varepsilon\geq0$
are arbitrary constants.

Given the first-passage time
\[
\tau_\ell=\inf\{t\geq0\dvtx  X_t\geq\ell\}
\]
for some arbitrary known level $\ell\geq0$, we aim at finding a
stopping time
with respect to the observation process $Y$ that best tracks $\tau
_\ell$.
Specifically, we consider the optimization problem
\begin{equation}\label{oppb}
\inf_\eta\ex|\eta-\tau_\ell| ,
\end{equation}
where the infimum is over all stopping times $\eta$ defined
with respect to the natural filtration induced by the $Y$ process.

To avoid trivial situations, we restrict $\ell$ and $\varepsilon$ to
be strictly positive. When $\ell=0$ or $\varepsilon=0$, \eqref{oppb}
is equal to zero: for $\ell=0$, $\eta=0$ is optimal, and for
$\varepsilon=0$, $\eta=\tau_\ell$ is optimal.

Define the stopping time
\[
\etas\stackrel{\mathrm{def}}{=} \inf\{t\geq0\dvtx  \hat{X}_t\geq\ell\},
\]
where
\[
\hat{X}_0\defeq0\quad \mbox{and}\quad \hat{X}_t\defeq st+\frac{1}
{1+\varepsilon^2}(Y_t-st),  \qquad t\geq1,
\]
is the minimum mean square estimator of $X_t$ given observation $Y_t$.

The following theorem provides a nonasymptotic upper bound on
\eqref{oppb}:
\begin{thmm}[(Upper bound)]\label{ub} Given
$0<\varepsilon<\infty$, $0<s<\infty$ and $0<\ell<\infty$, we have
\begin{eqnarray}\label{upbest}
{\ex}|\etas-\tau_\ell| &\leq&\sqrt{\frac{{2\ell
\varepsilon
^{2}}}{{\pi
s^3(1+\varepsilon^2)}}}\nonumber\\
&&{}+
\sqrt{\frac{4\varepsilon^2}{s^2(1+\varepsilon^2)}}\Biggl[3
\biggl(\frac
{\ell}{2\pi
s^3}\biggr)^{1/4}
+3\sqrt{\frac{3}{s}}+ \sqrt{3s} +6\Biggr]\\
&&{}+ \frac{4}{s\sqrt{1+\varepsilon^{2}}}+\frac{4}{s}+4.\nonumber
\end{eqnarray}
\end{thmm}


The next theorem provides a nonasymptotic lower bound on
${\ex}|\eta(Y_0^\infty) -\tau_\ell|$ for any estimator $\eta
(Y_0^\infty
)$ of
$\tau_\ell$ that has access to the entire observation sequence
$Y_0^\infty$. The
function $Q(x)$ is defined as
\[
Q(x)\defeq\frac{1}{\sqrt{2\pi}}\int_x^\infty
\exp(-u^2/2)\,du .
\]

\begin{thmm}[(Lower bound)]\label{lb}
Given
$0<\varepsilon<\infty$, $0<s<\infty$ and $0<\ell<\infty$, and any
integer $n$ such that $1\leq n<\ell/s$,
\begin{eqnarray}\label{low}
\inf_{\eta(Y_0^{\infty})}{\ex}|\eta-\tau|&\geq&
\sqrt{\frac{2n\varepsilon^2}{\pi s^2(1+\varepsilon^2)}}\nonumber\\
&&{}-
\biggl({\frac{2n}{\pi^3s^6}}\biggr)^{1/4} -\sqrt{\frac{2(\ell
-sn)_+}{\pi s^3}}
- 2 - \frac{6}{s}\\
&&{}-(2n^{3/2}+n/s+n^{1/2}\ell/s)Q\bigl((\ell-sn)/\sqrt{n}\bigr)^{1/2} .\nonumber
\end{eqnarray}
\end{thmm}

When $n$ approaches $\ell/s$ and $\ell/s$ tends to infinity in a
suitable way, the upper and lower bounds \eqref{upbest} and
\eqref{low} become tight. The following result is an immediate
consequence of these bounds by considering $n$ of the form
$n=\lfloor\ell/s-(\ell/s)^q\rfloor$, $1/2<q<1$, in Theorem
\ref{lb}\footnote{$\lfloor x\rfloor$ denotes the largest integer not greater
than $x$.}:
\begin{thmm}[(Asymptotics)]\label{th3}
Let $q$ be a constant such that $1/2<q<1$. In the asymptotic
regime where
\begin{eqnarray*}
s\biggl(\frac{\ell}{s}\biggr)^{q-1/2}&\longrightarrow&
\infty ,\qquad \biggl(\frac{\ell}{s}\biggr)^{1-q}
\frac{\varepsilon^2}{1+\varepsilon^2} \longrightarrow\infty
,\\
s\ell\frac{\varepsilon^4}{(1+\varepsilon^2)^2}&\longrightarrow&
\infty ,
\end{eqnarray*}
we have
\begin{eqnarray}\label{uiz}
\inf_{\eta(Y_0^\infty)}{\ex}|\eta-\tau_\ell|&=&\bigl(1+o(1)\bigr){\ex
}|\etas
-\tau_\ell|
\nonumber
\\[-8pt]
\\[-8pt]
\nonumber
&=& \sqrt{ \frac{2\ell\varepsilon^2}{{\pi
s^3(1+\varepsilon^{2})}}} \bigl(1 + o(1)\bigr) .
\end{eqnarray}
In particular, the equalities in \eqref{uiz} hold in the limit $\ell
\rightarrow
\infty$ for fixed $0<\varepsilon<\infty$ and $0<s<\infty$.
\end{thmm}

Theorem~\ref{th3} says that the sequential
estimator $\eta^*_\ell$ does as well as the best
estimators with the
foreknowledge of the entire observation process
$Y$, asymptotically.\footnote{$\eta(Y_0^\infty)$ need not be a
stopping time according to our notational
conventions.} Part of the reason for this is
that $\tau_\ell$ concentrates around $\ell/s$.
Hence, restricting estimators to depend only on
finitely many observations induces no loss of
optimality, asymptotically.

Consider now the setting where $\sum_{i=1}^tV_i$ and
$\sum_{i=1}^tW_i$ are replaced by standard Wiener processes, that is,
with the $X$ and the $Y$ processes defined as
\begin{eqnarray*}
X\dvtx&&\quad  X_0=0,\qquad  X_t=B_t +s t\qquad  \mbox{for }   t> 0,\\
Y\dvtx&&\quad  Y_0=0, \qquad Y_t=X_t+\varepsilon N_t  \qquad\mbox{for }   t>0 ,
\end{eqnarray*}
where $\{B_t\}_{t> 0}$ and $\{N_t\}_{t> 0}$ are two independent
standard Wiener processes. The previous results easily extend to
the Wiener process setting. Indeed, the analysis is simpler than
for the Gaussian random walk setting as there is no excess over
the boundary (variously known as overshoot) for a Wiener process---the
value of a Wiener process the first
time it reaches a certain level is equal to this
level.

Theorems~\ref{b1},~\ref{b2} and~\ref{b3} are analogous to
Theorems~\ref{ub},~\ref{lb} and~\ref{th3}, respectively.
\begin{thmm}[(Upper bound, Wiener process)]\label{b1}
Given
$0<\varepsilon<\infty$, $0<s<\infty$ and $0<\ell<\infty$, we have
\begin{equation}
{\ex}|\etas-\tau_\ell| \leq \sqrt{\frac{{2\ell
\varepsilon
^{2}}}{{\pi
(1+\varepsilon^2)s^3}}}+
\sqrt{\frac{36\varepsilon^2}{(1+\varepsilon^2)s^2}}\biggl(\frac
{\ell
}{2\pi
s^3}\biggr)^{1/4} .
\end{equation}
\end{thmm}

\begin{thmm}[(Lower bound, Wiener process)]\label{b2}
Given
$0<\varepsilon<\infty$, $0<s<\infty$, $0<\ell<\infty$, and $n$
such that
$0<n<\ell/s$, we have
\begin{eqnarray*}
\inf_{\eta(Y_0^{\infty})}{\ex}|\eta-\tau|&\geq&
\sqrt{\frac{2n\varepsilon^2}{\pi s^2(1+\varepsilon^2)}}-
\biggl({\frac{2n}{\pi^3s^7}}\biggr)^{1/4} -\sqrt{\frac{2(\ell
-sn)_+}{\pi s^3}}
\\
&&{}-(2n^{3/2}+n/s+n^{1/2}\ell/s)Q\bigl((\ell-sn)/\sqrt{n}\bigr)^{1/2} .
\end{eqnarray*}
\end{thmm}

The following theorem is an immediate consequence of
Theorems~\ref{b1} and~\ref{b2}.
\begin{thmm}[(Asymptotics, Wiener process)]\label{b3}
Theorem ~\ref{th3} is also valid in the Wiener process setting.
\end{thmm}

When there is no drift, that is, $s=0$, it turns out that \eqref{oppb} is
infinite for all $\ell>0$ and $\varepsilon>0$.
In fact, Theorem~\ref{prop1} below, which is
valid in both the Gaussian random walk and the
Wiener process settings, provides a stronger
statement:
\begin{thmm}\label{prop1} Let $s=0$, $0<\varepsilon<\infty$ and
$\ell>0$, and let $f(x)$, $x \geq0$, be a nonnegative
and nondecreasing function such that
\begin{equation}\label{eqprop1}
\ex f(\tau_{h}/2) = \infty
\end{equation}
for some constant $h>0$. Then,
\begin{longlist}[(ii)]
\item[(i)]
$\ex f(|\tau_\ell- \eta(Y_{0}^{\infty})|) = \infty$ for any
estimator $
\eta(Y_{0}^{\infty})$.
\item[(ii)]
If $f(x) = x^{p}$, $p \geq1/2$, then \eqref{eqprop1} holds for
all $h>0$. Hence,
\[
\ex
|\tau_\ell- \eta|^p=\infty
\]
for any estimator $
\eta(Y_{0}^{\infty})$ of $\tau_\ell$ whenever $p\geq1/2$.
\end{longlist}
\end{thmm}

A heuristic justification for Theorem~\ref{prop1}, claim (ii) is as
follows. When $s=0$,
$\ex\tau_\ell=\infty$ for any $\ell>0$. So, when $s=0$, it is
likely that
$\tau_\ell$ takes
some very large value. When this happens, the
estimate of $\tau_\ell$ is poor because of
the noise in the observation process whose
variance grows proportionally with time.

\section{Proofs of results}\label{sec:proofs}
In this section we prove Theorems~\ref{ub},~\ref{lb} and
\ref{prop1}. Theorems~\ref{b1} and~\ref{b2} are proved
in the same way as Theorems~\ref{ub} and~\ref{lb} by merely
ignoring overshoots.

The proofs of Theorems~\ref{b1} and~\ref{b2} are
therefore omitted.

In this section, $V$ and $W$ always denote
standard Gaussian random variables.

Before proving Theorems~\ref{ub},~\ref{lb} and~\ref{prop1}, we
establish a few auxiliary results related to overshoot estimates. These
results are based on the following theorem, given in~\cite{Mog1}, Theorem 2, equation~$(7)$,
which provides an upper bound on overshoot which is uniform
in the threshold level $\ell$.
\begin{thmm}[(\cite{Mog1})]\label{mogu}
Let $Z_{1},Z_{2},\ldots$ be i.i.d. random variables such that $\ex
Z_1 \geq0$. Define $S_{t} = Z_{1}+Z_2 + \cdots+ Z_{t}$, $\mu_\ell=
\inf\{t\geq1\dvtx  S_{t} \geq\ell\}$, and the overshoot $O_{\mu_\ell} =
S_{\mu_\ell} -\ell$.
Then,
\[
\sup_{\ell\geq0}{\ex}(O_{\mu_\ell}^{p}) \leq\frac{2(p+2)}{(p+1)}
\frac{{\ex}|Z_{1}|^{p+2}}{\ex(Z_{1}^{2})} \qquad \mbox{for all }
 p
> 0.
\]
\end{thmm}

Overshoot has been extensively studied and various other bounds
have been exhibited (see, e.g.,~\cite{Lo1,Gut,Ch}). However, to the
best of our knowledge, the bound given by Theorem~\ref{mogu} has not
been improved for
all $s\geq0$ and $p>0$. In particular, it is tighter than
Lorden's bound~\cite{Lo1} for small values of~$s$.

\begin{cor}\label{corwa}
Let $Z_{1},Z_{2},\ldots$ be i.i.d. random variables according to a
mean $s> 0$ and variance $\sigma^2\geq0$ Gaussian
distribution, and let $S_{t}$, $\mu_\ell$ and~$O_{\mu_\ell}$ be
defined as in Theorem~\ref{mogu}. Then,
\begin{equation}
\sup_{\ell\geq0}{\ex}(O_{\mu_\ell}) \leq2s + 4\sigma ,
\label{mog1g}
\end{equation}
and
%
\begin{equation}\label{stat1}
\frac{\ell}{s}\leq\frac{1}{s}\ex S_{\mu_\ell}=\ex\mu_\ell
\leq\frac{\ell}{s}+2+\frac{4 \sigma}{s} .\vspace*{-2pt}
\end{equation}
\end{cor}

\begin{pf}
Since
\[
{\ex}(Z_1)^{2} = s^{2} + \sigma^{2} \quad \mbox{and}\quad
{\ex}|Z_1|^{4} = {\ex}(s+\sigma V)^{4} = s^{4} + 6s^{2}\sigma^{2}
+ 3\sigma^{4}  ,
\]
we have
\[
\sup_{\ell\geq0}{\ex}(O_{\mu_\ell}^{2}) \leq\frac{8}{3}\biggl[
s^{2} +
5\sigma^{2} - \frac{2\sigma^{4}}{s^{2} + \sigma^{2}}\biggr],
\]
from Theorem~\ref{mogu} with $p = 2$. Therefore,
\begin{eqnarray*}
\sup_{\ell\geq0}{\ex}(O_{\mu_\ell}) & \leq&
\sqrt{\sup_{\ell\geq0}{\ex}(O_{\mu_\ell}^{2})} \\
&\leq&\sqrt{\frac{8}{3}\biggl[ s^{2} + 5\sigma^{2} -
\frac{2\sigma^{4}}{s^{2} +
\sigma^{2}}\biggr]}\\
&\leq&2s + 4\sigma ,
\end{eqnarray*}
which gives \eqref{mog1g}.

Now $\ex S_{\mu_\ell}= s \ex\mu_\ell$ by Wald's
equation since $0<s<\infty$ and $\ex\mu_\ell<\infty$. Hence, since
\[
\ell\leq\ex
S_{\mu_\ell}\leq\ell+\sup_{\ell\geq
0}{\ex}(O_{\mu_\ell}) ,
\]
inequality \eqref{stat1} follows from~\eqref{mog1g}.
\end{pf}

\begin{lem} \label{stat2cc} $\!\!\!$The following inequalities hold for all
$0\!<\!s\!<\!\infty$ and \mbox{$0\!<\!\ell\!<\!\infty$}:
\begin{eqnarray}\label{ihwi}
\ex(\ell/s - \tau_\ell)_{+} &\leq&\ex(\tau_\ell- \ell/s)_{+}
\leq
\sqrt{\frac{\ell}{{2\pi s^3}}} + 1 + \frac{3}{s} ,
\\
\label{stat2c}
{\ex}|\tau_\ell- \ell/s| &\leq&\sqrt{\frac{{2\ell
}}{{\pi
s^3}}} +
2 + \frac{6}{s} ,
\\
\label{stat2e}
{\ex}(X_{\tau_\ell} - s\tau_\ell)_{+} &\leq&
\sqrt{\frac{\ell}{2\pi s}} + 3s
+ 7.
\end{eqnarray}
\end{lem}

\begin{pf}
From Wald's equation $\ex X_{\tau_\ell}= s
\ex\tau_\ell$, since $0<s<\infty$ and $\ex\tau_\ell<\infty$, hence
$\ell\leq\ex X_{\tau_\ell}= s
\ex\tau_\ell$. Therefore, using the identity
$x=x_+-(-x)_+$,\footnote{$x_+\defeq\max\{0,x\}$.} we get
\[
0\leq{\ex}(\tau_\ell- \ell/s) = {\ex}(\tau_\ell- \ell/s)_{+}
- {\ex}(\ell/s - \tau_\ell)_{+} ,
\]
that is,
\begin{equation}\label{stat2b}
\ex(\ell/s - \tau_\ell)_{+} \leq\ex(\tau_\ell- \ell/s)_{+}  .
\end{equation}
We upper bound the right-hand side of \eqref{stat2b} as\footnote
{$\lceil x\rceil$
denotes the smallest integer not smaller
than $x$.}
\begin{eqnarray}\label{aga0}
{\ex}(\tau_\ell- \ell/s)_{+} &\leq&{\ex}(\tau
_\ell-
\cls)_{+}+1\nonumber\\
&=& {\ex}\bigl(\tau_\ell- \cls; \tau_\ell>\cls,X_\cls<\ell
\bigr)+1
\nonumber
\\[-8pt]
\\[-8pt]
\nonumber
&=& {\ex}\bigl(\nu_\ell- \cls; X_{\cls}< \ell\bigr)+1
\\
&=& {\ex}(\nu_G; G>0)+1 ,\nonumber
\end{eqnarray}
where $\nu_\ell\defeq\inf\{t\geq\cls\dvtx X_t\geq
\ell\}$ and $G \defeq\ell
-X_{\cls}$.

Since $G\leq-\sum_{i=1}^{\cls}V_{i}\stackrel{\mathrm{d}}{=}\sqrt
{\cls}
V$, using
equation \eqref{stat1} of Corollary~\ref{corwa} with $\sigma^2=1$ yields
\begin{eqnarray}\label{aga1}
{\ex}(\nu_{G}; G> 0)&\leq&{\ex} \biggl[\frac{G}{s}
+ 2 + \frac{4}{s};G > 0
\biggr] \nonumber\\
&\leq&{\ex} \biggl[\frac{\sqrt{\lceil\ell/s\rceil}V}{s}
+ 2 + \frac{4}{s};V > 0
\biggr]
\nonumber
\\[-8pt]
\\[-8pt]
\nonumber
&\leq&\sqrt{\frac{\cls}{s^{2}}}{\ex}(V)_{+} +1 + \frac
{2}{s}\\
&\leq&\sqrt{\frac{{\ell}}{{2\pi s^3}}} + 1 + \frac{3}{s}.\nonumber
\end{eqnarray}
From \eqref{stat2b}, \eqref{aga0} and \eqref{aga1} we get
\begin{equation}\label{ihw}
\ex(\ell/s - \tau_\ell)_{+} \leq\ex(\tau_\ell- \ell/s)_{+}
\leq
\sqrt{\frac{\ell}{{2\pi s^3}}} + 1 + \frac{3}{s} ,
\end{equation}
which gives \eqref{ihwi}.

Inequality \eqref{stat2c} is an immediate consequence of
\eqref{ihwi}.

Since $X_{\tau_\ell}\geq\ell$, we have
\[
{\ex}(X_{\tau_\ell}/s - \tau_\ell)_{+} \leq
{\ex}(X_{\tau_\ell}/s -\ell/s) + {\ex}(\ell/s -
\tau_\ell)_{+} .
\]
This, together with \eqref{ihw} and the
inequality
\begin{equation}\label{stat2f}
{\ex}(X_{\tau_\ell}/s - \ell/s) \leq2 + 4/s
\end{equation}
obtained from equation~\eqref{stat1} of
Corollary~\ref{corwa}, establishes \eqref{stat2e}.
\end{pf}

\begin{pf*}{Proof of Theorem~\protect\ref{ub}}
We prove Theorem~\ref{ub} by considering estimators of the form
\[
\eta^{(c)}=\inf\bigl\{t\geq1\dvtx  \hat{X}_{t}^{(c)} \geq
\ell\bigr\}  ,
\]
where $\hat{X}$ is defined as
\[\hat{X}_{0}^{(c)}=0,\qquad \hat{X}_{t}^{(c)} = st +
c(Y_{t} -
st) = st + c\Biggl[\sum_{i=1}^{t} V_{i} + \varepsilon
\sum_{i=1}^{t}W_{i}\Biggr],\qquad
t\geq1,
\]
for some constant $c\geq0$. We upper bound $\ex|\eta^{(c)}-\tau
_\ell|$,
$c\geq0$, and show that the optimal value of $c$ is
$1/(1+\varepsilon^2)$, which shall prove the theorem.

For $c = 0$, we have $\eta^{(0)}= \cls$, and equation~\eqref{stat2c}
of Lemma~\ref{stat2cc} gives
\begin{equation}\label{0estim}
{\ex}\bigl| \eta^{(0)}-\tau_\ell\bigr| \leq
\sqrt{\frac{2{\ell}}{{\pi
s^3}} }+ 3+ \frac{6}{s}.
\end{equation}

We now bound $\ex|\eta^{(c)}-\tau_\ell|$ for arbitrary values of
$c\geq0$. Since
\[
|x|=2x_+-x ,
\]
we have
\begin{equation}\label{piz}
{\ex}\bigl|\eta^{(c)} - \tau_\ell\bigr| = 2{\ex}\bigl(\eta
^{(c)} -
\tau_\ell\bigr)_{+} - {\ex}\bigl(\eta^{(c)} - \tau_\ell\bigr).
\end{equation}
Applying equation~\eqref{stat2c} of Corollary~\ref{corwa} to $\tau
_\ell
$ and $\eta$ yields
\[
\label{nuA}
{\ex}\bigl(\eta^{(c)} - \tau_\ell\bigr)\geq-2-\frac{4}{s}  ,
\]
hence from~\eqref{piz}
%
\begin{equation}\label{genup1}
{\ex}\bigl|\eta^{(c)} - \tau_\ell\bigr| \leq2{\ex}\bigl(\eta^{(c)}
- \tau_\ell\bigr)_{+} + 2+\frac{ 4}{s}.
\end{equation}
Below, we upper bound $\ex(\eta^{(c)}-\tau_\ell)_+$ then use
\eqref{genup1} to deduce a bound on ${\ex}|\eta^{(c)} -
\tau_\ell|$.

For notational convenience, throughout the
calculations we some-\break times~omit~the superscript ${(c)}$ and simply write $\hat{X}_t$ and
$\eta$ in place of $\hat{X}_t^{(c)}$ and~$\eta^{(c)}$.

Let us introduce the auxiliary stopping time
\[
\nu\defeq\inf\{t \geq\tau_\ell\dvtx  \hat{X}_{t} \geq\ell\}.
\]
It follows that
\begin{eqnarray}
\label{expaa} {\ex}(\eta- \tau_\ell)_{+} &\leq&
{\ex}(\nu- \tau_\ell; \eta> \tau_\ell) \nonumber\\
&\leq&{\ex}(\nu- \tau_\ell; \hat{X}_{\tau_\ell}\leq\ell
)\\
&=&\frac{1}{s}{\ex}(\hat{X}_{\nu} - \hat{X}_{\tau_\ell};
\hat{X}_{\tau_\ell}\leq\ell),\nonumber
\end{eqnarray}
where the second inequality holds since $\{\eta> \tau_\ell\}
\subseteq
\{\hat{X}_{\tau_\ell}\leq\ell\}$ and where for the last equality we
used Wald's
equation since $0<s<\infty$ and both $\nu$ and~$\tau_\ell$ have finite expectation.

Since the random walk $\hat{X}$ has incremental steps with mean $s$
and variance
$c^2(1+\varepsilon^2)$, from
equation~\eqref{mog1g} of
Corollary~\ref{corwa} and the strong Markov property of
$\hat{X}$ at time $\tau_\ell$, we get
\begin{eqnarray*}
{\ex}(\hat{X}_{\nu} - \hat{X}_{\tau_\ell}; \hat{X}_{\tau
_\ell}\leq
\ell
) &\leq&{\ex} \bigl[\ell+ 2s +
4c\sqrt{1+\varepsilon^{2}}-\hat{X}_{\tau_\ell}; \hat{X}_{\tau
_\ell}
\leq\ell\bigr]\\
&\leq&{\ex} \bigl[{X}_{\tau_\ell} + 2s +
4c\sqrt{1+\varepsilon^{2}}-\hat{X}_{\tau_\ell}; \hat{X}_{\tau
_\ell}
\leq X_{\tau_\ell}\bigr]\\
&\leq&{\ex} ({X}_{\tau_\ell}
-\hat{X}_{\tau_\ell})_++s + 2c\sqrt{1+\varepsilon^{2}} ,
\end{eqnarray*}
hence from \eqref{expaa}
\begin{equation}\label{bg}
\ex\bigl(\eta^{(c)} - \tau_\ell\bigr)_{+}\leq\frac{1}{s}{\ex}
\bigl({X}_{\tau_\ell}^{(c)} -\hat{X}_{\tau_\ell}\bigr)_++ \frac{s +
2c\sqrt{1+\varepsilon^{2}}}{s}.
\end{equation}
Before we compute an upper bound on
$\ex(X_{\tau_\ell}-\hat{X}_{\tau_\ell}^{(c)})_+$ for general
values of
$c\geq0$, we consider the
case $c=1$.

\textit{Case} $c=1$: We have $\hat{X}_{t}^{(1)} = Y_{t}$ and
$\eta^{(1)} = \inf\{t\geq0\dvtx  Y_{t} \geq\ell\}$.
Since
$Y_t\stackrel{\mathrm{d}}{=} X_t+\varepsilon\sqrt{t}W$ with $W$
independent of $X_t$, it follows that
\begin{eqnarray}\label{bhu}
{\ex} ({X}_{\tau_\ell} -\hat{X}_{\tau_\ell})_+&=&\ex\bigl(\varepsilon
\sqrt{\tau_\ell}W\bigr)_+\nonumber\\
&=&\varepsilon\ex\bigl(\sqrt{\tau_\ell}\bigr)\ex(W)_+\nonumber\\
&=&\frac{\varepsilon}{\sqrt{2\pi}}\ex\bigl(\sqrt{\tau_\ell}\bigr)
\\
&\leq&\frac{\varepsilon}{\sqrt{2\pi}}\sqrt{\ex({\tau_\ell
})}\nonumber\\
&\leq&\frac{\varepsilon}{\sqrt{2\pi}}\sqrt{\frac{\ell+2s+4}{s}},\nonumber
\end{eqnarray}
where for the first inequality we used Jensen's inequality, and
where the second inequality follows from equation \eqref{stat1}
of Corollary~\ref{corwa}.

Combining~\eqref{bhu} with \eqref{bg} ($c=1$) yields
\[
{\ex}\bigl(\eta^{(1)} - \tau_\ell\bigr)_{+}\leq
\varepsilon\sqrt{\frac{\ell+ 2s
+ 4}{2\pi s^3} }+\frac{s+2\sqrt{1+\varepsilon^{2}}}{s}
\]
which, together with \eqref{genup1}, gives
\begin{equation}\label{1stestim}
{\ex}\bigl|\eta^{(1)} -\tau_\ell\bigr| \leq2\varepsilon\sqrt
{\frac
{\ell+ 2s
+ 4}{2\pi s^3} } +
\frac{4(s+1+\sqrt{1+\varepsilon^{2}})}{s}.
\end{equation}
Comparing \eqref{1stestim} with \eqref{0estim}, we note that for
fixed $s>0$, if $\varepsilon\ll1$, then ${\ex}|\eta^{(1)}
-\tau_\ell| \ll{\ex}|\eta^{(0)} -\tau_\ell|$
for large
values of~$\ell$.

\textit{General case} $c\geq0$: We compute a general
upper bound on $\ex({X}_{\tau_\ell} -\hat{X}_{\tau_\ell}^{(c)})_+$,
$c\geq0$, and use \eqref{genup1} and \eqref{bg} to obtain an upper bound
on $\ex|\eta^{(c)}-\tau_\ell|$.

Let $U_i$ be the increment of the
random walk $Z_t={X}_t-\hat{X}_t^{(c)}$, that is,
\[
U_i=Z_i-Z_{i-1}=
(1-c)V_{i}-c\varepsilon W_{i}.
\]
Given the fixed time horizon $m =
\bls$, we have
\begin{equation}\qquad
{X}_{\tau_\ell}-\hat{X}_{\tau_\ell}^{(c)} = \sum_{i=1}^{m}U_{i} -
\openone\{\tau_\ell<m\}\sum_{i=\tau_\ell+1}^{m}U_{i}+
\openone\{\tau_\ell>m\} \sum_{i=m+1}^{\tau_\ell}U_{i}  ,
\end{equation}
and therefore
\begin{eqnarray}\label{sumpo}
\ex\bigl( {X}_{\tau_\ell}- \hat{X}_{\tau_\ell}^{(c)}\bigr)_+ &\leq&
\ex\Biggl(\sum_{i=1}^{m}U_{i}\Biggr)_+ + \ex\Biggl( -
\openone\{\tau_\ell<m\}\sum_{i=\tau_\ell+1}^{m}U_{i}
\Biggr)_+
\nonumber
\\[-8pt]
\\[-8pt]
\nonumber
&&{}+\ex\Biggl( \openone\{\tau_\ell>m\} \sum_{i=m+1}^{\tau_\ell
}U_{i}\Biggr)_+.
\end{eqnarray}
We bound each term on the right-hand side of \eqref{sumpo}. For the
first term, since $\sum_{i=1}^{m}U_{i} \stackrel{\mathrm{d}}{=}
\sqrt{m[(1-c)^{2} + c^{2}\varepsilon^{2}]}V$, we have
\begin{eqnarray}\label{ugo3}
{\ex}\Biggl(\sum_{i=1}^{m}U_{i}\Biggr)_{+} &=& \sqrt{m[(1-c)^{2} +
c^{2}\varepsilon^{2}]}\ex(V)_+
\nonumber
\\[-8pt]
\\[-8pt]
\nonumber
&= &\sqrt{\frac{m[(1-c)^{2} + c^{2}\varepsilon^{2}]}{2\pi}}\leq\sqrt{\frac{\ell[(1-c)^{2} + c^{2}\varepsilon^{2}]}{2\pi s
}} .
\end{eqnarray}
For the second term on the right-hand side of \eqref{sumpo}, since
$\tau_\ell$ is independent of $U_{\tau_\ell+1},U_{\tau_\ell
+2},\ldots,$ we have
\begin{eqnarray}\label{dzp}
\ex\Biggl( - \openone\{\tau_\ell<m\}\sum_{i=\tau_\ell
+1}^{m}U_{i}\Biggr)_+ &=&
{\ex}\bigl[\sqrt{(m- \tau_\ell)_{+}[(1-c)^{2} +
c^{2}\varepsilon^{2}]} V_{+} \bigr]\nonumber\\
& =& \sqrt{\frac{{(1-c)^{2} + c^{2}\varepsilon^{2}}}{{2\pi}}}
{\ex}\sqrt{(m- \tau_\ell)_{+}}
\nonumber
\\[-8pt]
\\[-8pt]
\nonumber
& \leq&\sqrt{\frac{[(1-c)^{2} + c^{2}\varepsilon^{2}]}{2\pi}
{\ex}(m- \tau_\ell)_{+}} \\
&\leq&\sqrt{\frac{[(1-c)^{2} + c^{2}\varepsilon^{2}]}{2\pi}
\Biggl[\sqrt{\frac{\ell}{{2\pi s^3}}} + 1 + \frac{3}{s}\Biggr]},\nonumber
\end{eqnarray}
where the first inequality holds by Jensen's inequality and where
the last inequality follows from
equation~\eqref{ihwi} of Lemma~\ref{stat2cc}.

For the third term on the right-hand side of \eqref{sumpo}, we have
\begin{eqnarray}\label{gk}\qquad
\ex\Biggl( \openone\{\tau_\ell>m\} \sum_{i=m+1}^{\tau_\ell
}U_{i}\Biggr)_+
&\leq&
c\varepsilon{\ex}
\Biggl(\openone\{\tau_\ell>m\}\sum_{i=m+1}^{\tau_\ell}W_{i}
\Biggr)_{+}
\nonumber
\\[-8pt]
\\[-8pt]
\nonumber
&&{}+(1-c)_+{\ex}
\Biggl(\openone\{\tau_\ell>m\}\sum_{i=m+1}^{\tau_\ell}V_{i}\Biggr)_{+}.
\end{eqnarray}
Since $\tau_\ell$ and $\{W_{i}\}$ are independent, we have
\[
\openone\{\tau_\ell>n\}\sum_{i=m+1}^{\tau_\ell}W_{i} \stackrel
{\mathrm{d}}{=}
\sqrt{(\tau_\ell-m)_{+}}W ,
\]
and a similar calculation as
for~\eqref{dzp} shows that
\begin{equation}\label{gk1}
{\ex} \Biggl[\openone\{\tau_\ell>m\}\sum_{i=m+1}^{\tau_\ell
}W_{i}\Biggr]_{+}
\leq\sqrt{\frac{1}{2\pi} \Biggl[\sqrt{\frac{\ell}{{2\pi s^3}}} + 2
+ \frac{3}{s}\Biggr]}.
\end{equation}
We now focus on the second expectation on the right-hand side
of~\eqref{gk}. Note first that, on $\{\tau_\ell> m\}$, we have
\[
\sum_{i=m+1}^{\tau_\ell}V_{i} = (X_{\tau_\ell} - X_{m}) - s(\tau
_\ell- m).
\]
Therefore, to bound $ {\ex}
(\openone\{\tau_\ell>m\}\sum_{i=m+1}^{\tau_\ell}V_{i}
)_{+}$, we
consider the ``shifted'' sequence $\{S_{t} = X_{t} - X_{m}\}_{t\geq
m}$, and its crossing of level $\ell- X_m$. Using \eqref{stat2e}
(with $\ell- X_{m}$ instead of $\ell$) we have
\begin{eqnarray}\label{gk3}
&&{\ex} \Biggl(\openone\{\tau_\ell>m\}\sum_{i=n+1}^{\tau_\ell
}V_{i}\Biggr)_{+}\nonumber\\
&&\qquad\leq{\ex}\bigl([X_{\tau_\ell} - X_{m} - s(\tau_\ell-
m)]_{+};X_{m} \leq
\ell\bigr)\nonumber\\
&&\qquad \leq{\ex} \sqrt{\frac{(\ell- X_{m})_{+}}{2\pi s}} + 3s + 7
\\
&&\qquad\leq\sqrt{\frac{{\ex}(\ell- X_m)_{+}}{2\pi s}} + 3s + 7
\nonumber\\
&&\qquad\leq\frac{\ell^{1/4}}{(2\pi s)^{3/4}} +\frac{1}{\sqrt{2\pi s}}+
3s +
7 ,\nonumber
\end{eqnarray}
where the third inequality follows from Jensen's inequality.
Combining~\eqref{gk} together with \eqref{gk1} and \eqref{gk3}
yields
\begin{eqnarray}\label{gk2}
&&\ex\Biggl( \openone\{\tau_\ell>m\} \sum_{i=m+1}^{\tau_\ell
}U_{i}
\Biggr)_+\nonumber\\
&&\qquad\leq
c\varepsilon\sqrt{\frac{1}{2\pi}
\Biggl[\sqrt{\frac{\ell}{{2\pi s^3}}} + 2 + \frac{3}{s}\Biggr]}
\\
&&\qquad\quad{}+(1-c)_+ \biggl(\frac{\ell^{1/4}}{(2\pi s)^{3/4}} +\frac{1}{\sqrt
{2\pi
s}}+ 3s + 7\biggr) ,\nonumber
\end{eqnarray}
and from \eqref{bg}, \eqref{sumpo}--\eqref{dzp} and
\eqref{gk2}, we get
\begin{eqnarray}\label{zh}
{\ex}\bigl(\eta^{(c)} - \tau_\ell\bigr)_{+} & \leq&
\sqrt{\frac{{\ell[(1-c)^{2} + c^{2}\varepsilon^{2}]}}{{2\pi
s^3}}}+c\varepsilon\sqrt{\frac{1}{2\pi s^2}
\Biggl[\sqrt{\frac{\ell}{{2\pi s^3}}} + 2 + \frac{3}{s}
\Biggr]}\nonumber
\\
&&{}+ \sqrt{\frac{[(1-c)^{2} + c^{2}\varepsilon^{2}]}{2\pi s^2}
\Biggl[\sqrt{\frac{\ell}{{2\pi s^3}}} + 1 + \frac{3}{s}\Biggr]}
\nonumber
\\[-8pt]
\\[-8pt]
\nonumber
&&{}+
\frac{(1-c)_+}{s}\biggl[\frac{\ell^{1/4}}{(2\pi s)^{3/4}} + \frac
{1}{\sqrt{2\pi s}}+3s + 7\biggr]\\
&&{}+1 + \frac{2c\sqrt{1+\varepsilon^{2}}}{s}.\nonumber
\end{eqnarray}
To minimize the first term on the right-hand side
of~\eqref{zh} (which is the dominant term as a
function of $\ell$), we set
$c = \bar{c}=1/(1+\varepsilon^{2})$ so as to minimize the factor
$(1-c)^{2} + c^{2}\varepsilon^{2}$. With $c=\bar{c}$ we have
$(1-c)^{2} + c^{2}\varepsilon^{2} =
\varepsilon^{2}/(1+\varepsilon^{2})$ and
$\eta^{(\bar{c})}=\etas$, hence, from \eqref{zh},
\begin{eqnarray*}
{\ex}(\etas- \tau_\ell)_{+} & \leq&
\sqrt{\frac{{\ell\varepsilon^{2}}}{{2\pi
(1+\varepsilon^2)s^3}}}+\frac{\varepsilon}{1+\varepsilon^2}
\sqrt{\frac{1}{2\pi s^2}
\Biggl[\sqrt{\frac{\ell}{{2\pi s^3}}} + 2 + \frac{3}{s}
\Biggr]}
\\
&&{}+ \sqrt{\frac{\varepsilon^{2}}{2\pi(1+\varepsilon^2)s^2}
\Biggl[\sqrt{\frac{\ell}{{2\pi s^3}}} + 1 + \frac{3}{s}\Biggr]}
\\
&&{}+
\frac{\varepsilon^2}{s(1+\varepsilon^2)}\biggl[\frac{\ell
^{1/4}}{(2\pi s)^{3/4}}
+ \frac{1}{\sqrt{2\pi s}}+3s + 7\biggr]\\
&&{}+1 + \frac{2}{s\sqrt{1+\varepsilon^{2}}}.
\end{eqnarray*}
Combining the second, third and fourth terms on the right-hand side of
the above
inequality, we get
\begin{eqnarray}\label{zh2}
{\ex}(\etas- \tau_\ell)_{+} & \leq&
\sqrt{\frac{{\ell\varepsilon^{2}}}{{2\pi
(1+\varepsilon^2)s^3}}}\nonumber\\
&&{}+
\frac{\varepsilon}{s\sqrt{1+\varepsilon^2}}\Biggl[3\biggl(\frac
{\ell}{2\pi
s^3}\biggr)^{1/4}
+3\sqrt{\frac{3}{s}}+ \sqrt{3s} +6\Biggr]\\
&&{}+ \frac{2}{s\sqrt{1+\varepsilon^{2}}}+1.\nonumber
\end{eqnarray}
Finally, combining \eqref{zh2} with \eqref{genup1} yields
\begin{eqnarray*}
{\ex}|\etas- \tau_\ell| &\leq&\sqrt{\frac{{2\ell\varepsilon
^{2}}}{{\pi
(1+\varepsilon^2)s^3}}}\nonumber\\
&&{}+
\frac{2\varepsilon}{s\sqrt{1+\varepsilon^2}}\Biggl[3\biggl(\frac
{\ell
}{2\pi
s^3}\biggr)^{1/4}
+3\sqrt{\frac{3}{s}}+ \sqrt{3s} +6\Biggr]\nonumber\\
&&{}+ \frac{4}{s\sqrt{1+\varepsilon^{2}}}+\frac{4}{s}+4,
\end{eqnarray*}
from which Theorem~\ref{ub} follows.
\end{pf*}

\begin{pf*}{Proof of Theorem~\protect\ref{lb}}
We prove Theorem~\ref{lb} by establishing a~lower bound on
${\ex}|\eta({Y_0^\infty})-\tau_\ell|$ for any estimator $\eta
({Y_0^\infty})$
that has access to the entire observation process $Y_0^\infty$.

Pick an arbitrary integer $n$ such that $1\leq n<\ell/s$. Then, we have
\begin{eqnarray}
\label{low31} \inf_{\eta(Y_0^{\infty})}{\ex}|\eta
-\tau_\ell| &=&\inf_{\eta(Y_0^{\infty})}{\ex}
\biggl|\biggl(\eta- n - \frac{\ell-X_{n}}{s}\biggr) +
\biggl(n+ \frac{\ell-X_{n}}{s} - \tau_\ell\biggr)\biggr|\nonumber
\\
& \geq&\inf_{\eta(Y_0^{\infty})}{\ex}
\biggl|\eta- n - \frac{\ell-X_{n}}{s}\biggr|- {\ex}\biggl|n+ \frac
{\ell
-X_{n}}{s} -
\tau_\ell\biggr| \\
&=& \frac{1}{s}\inf_{\eta(Y_0^{\infty})}{\ex}
|\eta- X_{n}| - {\ex}\biggl|n +
\frac{\ell-X_{n}}{s} - \tau_\ell\biggr|.\nonumber
\end{eqnarray}

The first expectation on the right-hand side of
\eqref{low31} is lower bounded as follows. Since $X_n$ and $ Y_n$ are
jointly Gaussian, we may represent $X_n$ as
\[
X_n\stackrel{\mathrm{d}}{=}
\sqrt{n\varepsilon^2/(1+\varepsilon^2)} V+c\cdot Y_n + d ,
\]
where $V$ is a standard Gaussian random variable independent of
$\{Y_n\}$, and where $c$ and $d$ are (nonnegative) constants (that
depend on $s$ and $\varepsilon$). Using this alternative
representation of $X_n$ yields
\begin{eqnarray}\label{gag3}
\inf_{\eta(Y_0^{\infty})}{\ex} |\eta-
X_{n}|&=&\inf_{\eta(Y_0^{\infty})}{\ex}
\bigl|\eta- c\cdot Y_n -
d-\sqrt{n\varepsilon^2/(1+\varepsilon^2)}
V\bigr|\nonumber\\
&=&\sqrt{\frac{n\varepsilon^2}{1+\varepsilon^2}}\inf_{\eta
(Y_0^{\infty})}{\ex}
|\eta- V|\nonumber\\
&=&\sqrt{\frac{n\varepsilon^2}{1+\varepsilon^2}}\inf_{e}{\ex}
|e - V|\\
&=&\sqrt{\frac{n\varepsilon^2}{1+\varepsilon^2}}{\ex}
|V|\nonumber\\
&=&\sqrt{\frac{2n\varepsilon^2}{\pi(1+\varepsilon^2)}} ,\nonumber
\end{eqnarray}
where the infimum on the right-hand side of the third equality is over
constant estimators (i.e., independent of
$Y_0^\infty$) since $V$ is independent of
$Y_0^\infty$, and where
for the fourth equality we used the fact that the median of
a~random variable is its best estimator with respect to the average
absolute deviation.

We now upperbound the second expectation on the right-hand side of~\eqref{low31}.
We have
\begin{eqnarray}\label{nils}
{\ex}\biggl|n+ \frac{\ell-X_{n}}{s} -
\tau_\ell\biggr|&=&{\ex}\biggl[\biggl|n+ \frac{\ell-X_{n}}{s} -
\tau_\ell\biggr|;\tau_\ell> n\biggr]
\nonumber
\\[-8pt]
\\[-8pt]
\nonumber
&&{}+
{\ex}\biggl[\biggl|n+ \frac{\ell-X_{n}}{s} -
\tau_\ell\biggr|;\tau_\ell\leq n\biggr] .
\end{eqnarray}
For the first term on the right-hand side of
\eqref{nils}, we use \eqref{stat2c} to
get
\begin{equation}\label{cav}
{\ex}\biggl[\biggl|n+ \frac{\ell-X_{n}}{s} -
\tau_\ell\biggr|\Big|X_{n},\tau_\ell>n\biggr] \leq
\sqrt{\frac{{2(\ell-X_{n})}}{{\pi s^3}}} + 2 + \frac{6}{s}
\end{equation}
on $\{X_{n} \leq\ell\}$.
Since $X_{n} \stackrel{\mathrm{d}}{=} sn + \sqrt{n} V$,
\begin{eqnarray*}
{\ex}(\ell-X_{n})_{+} &=& {\ex}\bigl(\ell-sn - \sqrt{n} V\bigr)_{+}\\
&\leq&\sqrt{n} {\ex}V_{+} +(\ell-sn)_+\\
&=&\sqrt{\frac{n}{2\pi}} +(\ell-sn)_+ .
\end{eqnarray*}
Hence, from Jensen's inequality
\[
{\ex}\sqrt{(\ell-X_{n})_{+}} \leq\sqrt{{\ex}(\ell-X_{n})_{+}}
\leq\Biggl(\sqrt{\frac{n}{2\pi}} +(\ell-sn)_+\Biggr)^{1/2},
\]
and therefore, by taking expectation on both
sides of \eqref{cav} we get
\begin{eqnarray}
\label{low31a}
&&{\ex}\biggl[\biggl|n+ \frac{\ell-X_{n}}{s} -
\tau_\ell\biggr|; \tau_\ell>n\biggr]
\nonumber
\\[-8pt]
\\[-8pt]
\nonumber
&&\qquad \leq\sqrt{\frac{2}{{\pi s^3}}}
\Biggl(\sqrt{\frac{n}{2\pi}} +(\ell-sn)_+\Biggr)^{1/2} + 2 + \frac{6}{s}.
\end{eqnarray}

For the second term on the right-hand side of \eqref{nils},
\begin{eqnarray}\label{bno}
&&{\ex}\biggl[\biggl|n+ \frac{\ell-X_{n}}{s} -
\tau_\ell\biggr|;\tau_\ell\leq n\biggr]\nonumber\\
&&\qquad\leq(n+\ell/s)\pr(\tau
_\ell
\leq n)+(1/s)\ex(|X_n|;\tau_\ell\leq
n)\nonumber\\
&&\qquad\leq(n+\ell/s)\pr(\tau_\ell\leq n)+(1/s)\bigl( \ex(X_n)^2\pr
(\tau
_\ell\leq n)\bigr)^{1/2}\\
&&\qquad= (n+\ell/s)\pr(\tau_\ell\leq n)+(1/s)\bigl( (n+s^2n^2)\pr(\tau
_\ell
\leq n)\bigr)^{1/2}\nonumber\\
&&\qquad\leq\bigl(2n+\sqrt{n}/s+\ell/s\bigr)\pr(\tau_\ell\leq n)^{1/2} ,\nonumber
\end{eqnarray}
where the second inequality follows from the Cauchy--Schwarz inequality.
Further,
\begin{eqnarray*}
\pr(\tau_\ell\leq n)&=&\sum_{i=1}^n\pr(\tau_\ell=i)\\
&\leq&\sum_{i=1}^n\pr(X_i\geq\ell)\\
&\leq& n Q\bigl((\ell-sn)/\sqrt{n}\bigr) .
\end{eqnarray*}
Hence, from \eqref{bno},
\begin{eqnarray}\label{fui}
&&\ex\biggl[\biggl|n+ \frac{\ell-X_{n}}{s} -
\tau_\ell\biggr|;\tau_\ell\leq
n\biggr]
\nonumber
\\[-8pt]
\\[-8pt]
\nonumber
&&\qquad\leq(2n^{3/2}+n/s+n^{1/2}\ell/s)Q\bigl((\ell-sn)/\sqrt
{n}\bigr)^{1/2} .
\end{eqnarray}
Combining \eqref{low31}--\eqref{nils}, \eqref{low31a}
and~\eqref{fui}, we get
\begin{eqnarray*}
\inf_{\eta(Y_0^{\infty})}{\ex}|\eta
-\tau_\ell|&\geq&
\sqrt{\frac{2n\varepsilon^2}{\pi s^2(1+\varepsilon^2)}}\\
&&{}-
\biggl({\frac{2n}{\pi^3s^6}}\biggr)^{1/4} -\sqrt{\frac{2(\ell
-sn)_+}{\pi s^3}}
- 2 - \frac{6}{s}\\
&&{}-(2n^{3/2}+n/s+n^{1/2}\ell/s)Q\bigl((\ell-sn)/\sqrt{n}\bigr)^{1/2} ,
\end{eqnarray*}
yielding the desired result.\vadjust{\goodbreak}
\end{pf*}

\begin{pf*}{Proof of Theorem~\protect\ref{prop1}} We prove the result only for
the Gaussian random walk setting. The proof for the Wiener process
setting follows the same arguments and is therefore
omitted.

Let $s=0$ and fix $0<\varepsilon<\infty$ and $0<\ell<\infty$.
We show that given $h>0$,
\[
\inf_{\eta(Y_0^\infty)}\ex
f(|\eta-\tau_\ell|)\geq k\ex f(\tau_h/2)
\]
for
some strictly
positive constant $k$. Hence, if $\ex
f(\tau_h/2)=\infty$ for some $h>0$, then
$\inf_{\eta(Y_0^\infty)}\ex
f(|\eta-\tau_\ell|)=\infty$, which yields claim
(i).

The first step consists in removing the noise in the observation
process $Y$ from time $t=2$ onward; that is, instead of
$\{Y_t\}_{t\geq0}$, we consider the better observation process
$\{Z_t\}_{t\geq0}$ defined as
\begin{eqnarray*}
Z_{0} &=& 0, \\
Z_{1} &=& X_1+\varepsilon W_{1} =
V_{1} + \varepsilon W_{1}, \\
Z_{t} &=& X_{t} - X_{t-1} = V_{t}, \qquad  t \geq2.
\end{eqnarray*}
Clearly, it is easier to estimate $\tau_{\ell}$ based on $Z_0^\infty$
than based on $Y_0^\infty$; one gets ${Y}_{t}-Y_{t-1}$ by
artificially adding the ``noise'' $\varepsilon W_t$ to $Z_t$, $t\geq
1$. Therefore,
\begin{equation}\label{icap}\inf_{\eta(Y_{0}^{\infty}
)}{\ex
}f(| \eta-
\tau_\ell|) \geq
\inf_{\eta({Z}_{0}^{\infty})} \ex f(|\eta-\tau_\ell
|) .
\end{equation}
Given $Z_{0}^{\infty}$, estimation
errors on $\tau_\ell$ are only due to the unknown value of~$X_{1}$
because of the unknown value of the noise $\varepsilon W_1$. In
turn, given $Z_0^\infty$, it is sufficient to consider only
$Z_{1}$ in order to estimate $X_1$ ($Z_1$ is a sufficient
statistic for $X_1$).

Below, we are going to make use of the important property that the
conditional density function of $X_1(=V_1)$ given $Z_1$ is not
degenerated since it is given by
\[
p(x|z) = \frac{\sqrt{1+\varepsilon^{2}}} {\varepsilon
\sqrt{2\pi}}\exp\biggl\{- \frac{(1+\varepsilon^{2})}
{2\varepsilon^{2}}\biggl(x -
\frac{z}{1+\varepsilon^{2}}\biggr)^{2}\biggr\} ,
\]
and since $\varepsilon>0$ by assumption.

Define $C = C(Z_1)= Z_1/(1+\varepsilon^2)-h/2$ and $D =
D(Z_1)= Z_1/(1+\varepsilon^2)+h/2$ where $h>0$ is some
arbitrary constant. From the above nondegeneration property it
follows that
\[
\pr(X_1 \leq C ) =\pr(X_1 \geq
D) \defeq
\delta_1=\delta_1(h,\varepsilon)>0 .
\]
Using this, we lower bound
\[
\inf_{\eta({Z}_{0}^{\infty})}
\ex f(|\eta-\tau_\ell|)
\]
by considering the following three-hypothesis problem: with probability
$1 - 2\delta_1$, $X_{1}$ is
known exactly (hence $\tau_\ell$ is known exactly as well), and with
equal probability $\delta_1$, $X_1$ is either equal to $C$ or
equal to $D$ (and no additional information on $X_1$ is
available). More specifically, denoting by $\tau_\ell^{\tiny{C}}$ the
value of $\tau_\ell$ when $X_{1} = C$, and by $\tau_\ell^{{\tiny{D}}}$
the value of $\tau_\ell$ when $X_{1} = d$, we have
\begin{eqnarray}\label{uuyy}
&&\inf_{\eta(Z_{0}^{\infty})} \ex f(| \eta- \tau_\ell
|)\nonumber\\
&&\qquad\geq\inf_{\eta(Z_{0}^{\infty})} \{ \ex[f(|
\eta-\tau_\ell|);X_1\leq
C]+ \ex[f(|\eta-\tau_\ell|);X_1\geq D]\}\nonumber\\
&&\qquad\geq\inf_{\eta(Z_{0}^{\infty})}\{ \ex[f(|
\eta-\tau_\ell^C |);X_1\leq
C]+ \ex[f(|\eta-\tau_\ell^D |);X_1\geq D]\}\\
&&\qquad= \delta_1 \inf_{\eta({Z}_{0}^{\infty})} \ex[
f(|\eta- \tau_\ell^{C} |) +
f(|\eta- \tau_\ell^{D}|)] \nonumber\\
&&\qquad \geq\delta_1 {\ex}f\biggl(\frac{\tau_\ell^{C} -
\tau_\ell^{D}}{2}\biggr) ,\nonumber
\end{eqnarray}
where the second and third inequalities follow from the assumption
that~$f(x)$ is nonnegative and nondecreasing. Further, since
$\tau_{\ell}^C\stackrel{\mathrm{d}}{=}\tau_{(\ell-C)_+}$ and since
$\tau_{\ell_1}-\tau_{\ell_2} \stackrel{\mathrm{d}}{=} \tau_{\ell
_1-\ell_2} $,
$\ell_1\geq\ell_2$, from \eqref{uuyy} we get
\begin{eqnarray}
\label{uuyy2}
\inf_{\eta(Z_{0}^{\infty})} \ex f(|
\eta- \tau_\ell|)&\geq&\delta_1 {\ex}f\biggl(\frac{\tau_\ell
^{C} -
\tau_\ell^{D}}{2}\biggr) \nonumber\\
&=&\delta_1 \ex f\biggl(
\frac{\tau_{(\ell- C)_+} - \tau_{(\ell- D)_+}}{2}\biggr)
\\
& =& \delta_1 \ex f\biggl(
\frac{\tau_{(\ell-C)_+-(\ell-D)_+}}{2}\biggr) .\nonumber
\end{eqnarray}
Now, on $\{D\leq\ell\}$ we have
\[
(\ell-C)_+-(\ell-D)_+=D-C=h ,
\]
therefore from \eqref{uuyy2} we get
\begin{equation}\label{zui}
\inf_{\eta(Z_{0}^{\infty})} \ex f(| \eta- \tau_\ell
|)\geq\delta_1\delta_2 \ex f\biggl( \frac{\tau_{h}}{2}\biggr) ,
\end{equation}
where
\[
\delta_2\defeq\delta_2(h,l,\varepsilon)=\pr(D\leq\ell)>0 .
\]
Claim (i) follows from \eqref{zui} and \eqref{icap}.

We now prove claim (ii). Let $\{B_{t}\}_{t \geq
0}$ be the standard
Wiener process whose value at integer times
$t=0,1,2,\ldots$ corresponds to process $X$, and
let
\[
\tilde{\tau}_h \defeq\inf\{t\geq0\dvtx  B_{t}
=h\}.
\]
Since $\tilde{\tau}_h \leq\tau_h$ for all $h\geq
0$, had we proved that
${\ex}f(\tilde{\tau}_\ell/2) =
\infty$, equation
\[
\ex f(\tau_{h}/2) = \infty
\]
would hold since $f(x)$ is nondecreasing.

From the reflection principle we get
\[
{\pr}(\tilde{\tau}_h \leq t) = 2{\pr}(B_{t} \geq h) =
2Q\biggl(\frac{h}{\sqrt{t}}\biggr), \qquad  h>0,  t > 0 ,
\]
hence for $h>0$,
\begin{eqnarray*}
{\ex}f(\tilde{\tau}_h/2) &=& 2\int _{0}^{\infty}
f(t/2)\,dQ\biggl(\frac{h}{\sqrt{t}}\biggr) \\
&=&
\frac{h}{\sqrt{2\pi}}\int _{0}^{\infty}\frac{f(t/2)}{t^{3/2}}
 e^{-h^2/2t}\,dt \\
&>& \frac{he^{-h/2}}{\sqrt{2\pi}}
\int _{h}^{\infty}\frac{f(t/2)}{t^{3/2}} \,dt.
\end{eqnarray*}
Therefore, if $f(x) = x^{p}$ with $p\geq1/2$, then $
{\ex}f(\tilde{\tau}_h/2) = \infty$ for all $h>0$. Claim~(ii) follows.
\end{pf*}

\section{Concluding remarks} We considered the
problem of sequentially estimating a random walk
first-passage time through noisy observations. Non\-asymptotic upper and
lower bounds
on minimum mean absolute deviation have been derived that coincide in certain
asymptotic regimes.

Extensions to other loss functions or non-Gaussian
settings may be envisioned. For the latter, an
interesting problem is
the derivation of a good lower bound. In fact, a main step in the
proof of Theorem~\ref{lb} [see argument after
equation~(\ref{low31})] takes advantage of the fact that $X_n$ and
$Y_n$ are jointly Gaussian.

Finally, note that at least some of the presented
arguments apply to stopping times other than first-passage times since the
basic property that we used is that $\tau$
concentrates around its mean (assuming a positive
drift).

\printaddresses

\end{document}